\documentclass[11pt,letter]{amsart}
\usepackage{amsfonts,amssymb,amsthm,amsmath,amsxtra,amscd,verbatim,eucal}
\usepackage[all]{xy}
\usepackage[dvips]{graphics}
\usepackage{graphicx,subfig}

\theoremstyle{plain}
\newtheorem{thm}{Theorem}[section]
\newtheorem{lem}[thm]{Lemma}
\newtheorem{prop}[thm]{Proposition}
\newtheorem{cor}[thm]{Corollary}
\newtheorem{conj}[thm]{Conjecture}

\theoremstyle{definition}
\newtheorem{defi}[thm]{Definition}

\theoremstyle{remark}

\newtheorem{rmk}[thm]{Remark}

\def\C{{\mathbb C}}
\def\A{{\mathbb A}}
\def\R{{\mathbb R}}

\def\P{{\mathbb P}}

\def\E{\mathcal{E}}

\def\X{\mathfrak{X}}

\def\a{\alpha}
\def\b{\beta}
\def\g{\gamma}
\def\d{\delta}

\def\ff{\psi}

\def\ep{\epsilon}

\def\l{\lambda}

\def\om{\omega}
\def\p{\pi}

\def\D{\Delta}
\def\G{\Gamma}

\def\S{\Sigma}

\def\.{\cdot}
\def\~{\widetilde}
\def\^{\widehat}

\def\o{\circ}
\def\ov{\overline}
\def\rat{\dashrightarrow}

\def\de{\partial}

\renewcommand{\and}{ \ \, \text{and} \ \, }

\newcommand{\fall}{ \quad \text{for all} \ \, }

\DeclareMathOperator{\NE} {NE}
\DeclareMathOperator{\CNE} {\ov{\NE}}

\DeclareMathOperator{\Pic} {Pic}

\DeclareMathOperator{\mult} {mult}

\DeclareMathOperator{\Div} {Div}

\DeclareMathOperator{\Nef} {Nef}

\title{On the Mori cone of blow-ups of the plane}

\author{Tommaso de Fernex}
\address{Department of Mathematics, University of Utah, 155 South 1400 East,
Salt Lake City, UT 48112-0090, USA}
\email{defernex@math.utah.edu}

\thanks{The author was partially supported by NSF CAREER Grant DMS-0847059.}


\thanks{Version of October 23, 2009.}

\begin{document}

\begin{abstract}
We discuss some properties of the extremal rays of the cone of effective curves of
surfaces that are obtained by blowing up $\P^2$ at points in very general position.
The main motivation is to rectify an incorrect interpretation, in terms of 
the geometry of this cone, of the Segre--Harbourne--Gimigliano--Hirschowitz 
conjecture. Even though the arguments are based on elementary computations, 
the point of view leads to observe some other properties of the Mori cone, 
among which the fact that in the case of ten or more points
the cone is not countably generated. 
\end{abstract}

\maketitle

\section{Introduction}

There are two important open problems
on plane curves with assigned multiplicities at very general points
$x_1,\dots,x_r \in \P^2$ (stated here in progressive order of difficulty).
The first is a conjecture of Nagata \cite{Nag} saying that if $r \ge 10$, then
$$
\deg(D) \ge \frac 1{\sqrt r} \sum_{i=1}^r \mult_{x_i}(D)
$$
for every effective divisor $D$ in $\P^2$.\footnote{The original
formulation of Nagata only involves uniform conditions on the multiplicities
at the chosen points, does not assume that $D$ is irreducible, 
and predicts that the strict inequality holds.
It is however well-known, and in any case easy to see,
that the two formulations are equivalent.}
The second conjecture says that the stronger bound
\begin{equation}\label{eq:second-conj}\tag{$\dag$}
\deg(D)^2 \ge \sum_{i=1}^r \mult_{x_i}(D)^2.
\end{equation}
holds for every nonrational integral curve $D$ in $\P^2$.

These conjectures can be reformulated in terms of the geometry of the
Mori cone of the blow-up of $\P^2$ at the points $x_1,\dots,x_r$.
In general, if $X$ is a complex projective manifold,
the Mori cone $\CNE(X)$ of $X$ is defined as the closure of the convex cone
in $N_1(X)$ spanned by the numerical classes of effective curves.
When the canonical class $K = K_X$ of $X$ is not numerically trivial,
Mori's Cone Theorem (e.g., see \cite{KM}) states that the cone $\CNE(X)$
is generated by its {\it $K$-positive part} $\CNE(X) |_{K^{\ge 0}}$
and its {\it $K$-negative extremal rays} $R_i$:
$$
\CNE(X) = \CNE(X) |_{K^{\ge 0}}  + \sum R_i
$$
(here $K^{\ge 0}$ denotes the set of classes $\a \in N_1(X)$ with $K\.\a \ge 0$).
Every ray $R_i$ is associated to an extremal contraction,
namely, a proper morphism $X \to Y_i$ with connected fibers
that contracts precisely those curves on $X$ whose classes lie on $R_i$.
While the $K$-negative side of $\CNE(X)$ is related via Mori theory to
the geometry of $X$, little is known about the properties of
the $K$-positive side of the cone.

In this regard,
the special situation where $X$ is the blow-up of $\P^2$ at a set of very general points
offers a very interesting case of investigation.
In light of the equivalent conjectures of
Segre, Harbourne, Gimigliano and Hirschowitz \cite{Seg,Har,Gim,Hir89}
(henceforth SHGH Conjecture), it is expected that if
$X$ is the blow-up of $\P^2$ at a set $\S$ of $r \ge 10$ points in very general position,
then some portion of the positive part of the boundary of the Mori cone of $X$
is ``circular'' (that is, is supported upon a spherical cone).

More precisely, if $Q \subset N_1(X)$ denotes the quadric cone
consisting of classes $\a$ with $\a^2 \ge 0$ and $\a\.h \ge 0$
for some ample class $h$, then the SHGH Conjecture
has the following implication in terms of the geometry of the Mori cone of $X$.

\begin{conj}[$(-1)$-Curves Conjecture]
If $X$ is the blow-up of $\P^2$ at $r$ very general points, then
$$
\CNE(X) = Q + \sum R_i,
$$
where the sum is taken over all $K$-negative extremal rays $R_i$
of $\CNE(X)$.
\end{conj}

By Mori theory, if $r \ge 2$ then every $K$-negative extremal ray $R_i$
of $\CNE(X)$ is spanned by a $(-1)$-curve (namely, a smooth rational
curve $C_i$ with self intersection $C_i^2 = -1$).
The formula in the statement of the conjecture
is thus equivalent to saying that the $(-1)$-curves
are the only integral curves on $X$ with negative self-intersection,
hence the name of the conjecture.
A small step in this direction comes from the following fact.

\begin{prop}[\protect{\cite[Proposition~2.4]{dF}}]\label{prop:dF}
If $C$ is an integral rational curve on $X$ with negative self-intersection,
then $C$ is a $(-1)$-curve.
\end{prop}

This property implies in particular
that the $(-1)$-Curve Conjecture is equivalent to
the conjectural bound on degree and multiplicities
given in \eqref{eq:second-conj} for nonrational curves. This also explains why
the latter implies the Nagata Conjecture.

In view of these conjectures, one might hope that in fact
the whole $K$-positive part of the cone of curve of $X$ is
circular, namely, that
\begin{equation}\label{eq:misinterpretation}\tag{$\ddag$}
\CNE(X)|_{K^{\ge 0}} = Q|_{K^{\ge 0}}.
\end{equation}
This almost seems to be a consequence of the above conjectures.
An incorrect argument was given in \cite{dF} to deduce \eqref{eq:misinterpretation}
as an equivalent formulation of the $(-1)$-Curves Conjecture.\footnote{This
misinterpretation of the $(-1)$-Curve Conjecture,
which is due to an incorrect application
of the result from \cite{CP}, does not affect the
remaining part \cite{dF} where partial
results towards the $(-1)$-Curves Conjecture are obtained.}

The first motivation of this paper is to correct such misinterpretation
of the conjecture. By taking into more careful consideration the
location of the $K$-negative extremal rays, it turns out that \eqref{eq:misinterpretation}
can only hold true when $r \le 10$. The case in which $r\le 9$ is well-understood
(the equality holds in this case),
and we will show that if $r=10$ then \eqref{eq:misinterpretation}
would indeed follow from the $(-1)$-Curves Conjecture
(and hence from the SHGH Conjecture).
In fact, in general we will see that
if the $(-1)$-Rays Conjecture is true, then for every $r \ge 10$
$$
\CNE(X) |_{(K-sL)^{\ge 0}} = Q |_{(K-sL)^{\ge 0}},
$$
where $L$ is the pull-back of a line in $\P^2$ and $s = \sqrt{r-1}-3$.
This in particular shows that, assuming the conjecture,
some portion of the boundary of $\CNE(X)$ is indeed ``circular''.
On the other hand, we will see that there is always a strict inclusion
$$
\CNE(X) |_{K^{\ge 0}} \supsetneq Q|_{K^{\ge 0}}
$$
for every $r > 10$, independently of any conjecture.

By further analyzing the geography of the $K$-negative extremal rays $R_i$,
we deduce some other properties of the Mori cone.
In general, it is easy to see that
$\CNE(X)$ is contained in the closure of the
cone $\R_{\ge 0}[-K_X] + \sum R_i$.
When $r=10$, this implies that
$$
\CNE(X)|_{K^\perp} = Q|_{K^\perp}.
$$
Even if these facts are not enough to produce examples of irrational Seshadri constants,
they imply for instance that, for any $r \ge 10$,
both the Mori cone and the nef cone of $X$ are not countably generated.
This point of view should also be useful
to see the combinatorics of the extremal facets of $\CNE(X)$
spanned by $(-1)$-curves and study the Weyl group of $X$,
which was originally investigated by Du Val in \cite{DV}.

\subsection*{Acknowledgements}

The author would like to thank
Lawrence Ein
and
Antonio Laface
for very useful conversations.

\section{Very general position}\label{sect:very-gen-pos}

Generally speaking, a point of a variety is said to be
``very general'', or ``in very general position'',
if it is chosen in the complement of the countable
union of preassigned proper closed subsets. The notion is thus
relative to such a choice of conditions.

In this paper we restrict our attention to the
2-dimensional case, and adopt the following definition
(cf. \cite[Definition~2.1]{dF}).

\begin{defi}\label{defi:very-gen-pos}
Let $S$ be a smooth projective surface. A set $\S = \{x_1,\dots,x_r\}$
of $r \ge 1$ distinct points on $S$ is said to be {\it in very general position}
if the following condition is fulfilled:
for every integral curve $D \subset S$, the element
$(D, (x_1,\dots,x_r)) \in \Div(S) \times S^r$
belongs to an irreducible flat family
$$
\{ (D_t,(x_{1,t},\dots,x_{r,t})) \mid t\in T \} \subset \Div(S) \times S^r
$$
such that $\mult_{x_i,t}(D_t)$ and $p_g(D_t)$ are independent of $t$, and
the map $\ff \colon T \to S^{(r)}$ given by $\ff(t) = \{x_{1,t},\dots,x_{r,t}\}$
is dominant.
\end{defi}

\begin{rmk}\label{rmk:very-gen-pos:analytic}
In the above definition, one can assume that $T$ is an open disk in $\C^m$
(in the analytic category), in which case the last condition is
that the image of $\ff$ is dense in the Zariski topology.
\end{rmk}

\begin{rmk}
For every smooth projective surface $S$, the locus in $S^{(r)}$ parametrizing points
in very general position according to the above definition
is the complement of the union of countably many proper closed subsets,
and if $\S$ is a set in very general position, then so is every subset of $\S$.
(see \cite[Remark~2.2]{dF}).
\end{rmk}

In this paper we are interested in the case of $\P^2$.
In general, suppose that $\S$ is any
set of possibly infinitely near points of $\P^2$,
and let $f \colon X \to \P^2$ be the corresponding blow-up of $\P^2$.
Any other reduction (i.e., proper birational morphism)
$g \colon X \to \P^2$ gives rise to a birational
transformation $g \o f^{-1} \colon \P^2 \rat \P^2$.
Note that, up to an isomorphism, $g$ is the composition
of a sequence of blow-ups centered at smooth points
(cf. \cite[Theorem~II.11]{Be}).
We define the {\it transform} of $\S$ via $g\o f^{-1}$ to be the
set $\G$ of possibly infinitely near points needed to blow up
$\P^2$ to produce $g$:
$$
\xymatrix{
& X \ar[dl]_{\text{blowing-up $\S$}} \ar[dr]^{\text{blowing-up $\G$}}\\
\P^2 \ar@{-->}[rr]^{g\o f^{-1}} && \P^2.
}
$$

\begin{prop}\label{prop:very-gen-position}
Suppose that $\S\subset \P^2$ is a set of $r$ points in very general position, and
let $f \colon X \to  \P^2$ be the blow-up along $\S$.
Suppose that $g \colon X \to \P^2$ is any other reduction,
and let $\G$ be the transform of $\S$ via $g\o f^{-1}$.
Then $\G$ is a set of $r$ distinct points of $\P^2$ in very general position.
\end{prop}

\begin{proof}
By Proposition~\ref{prop:dF},
every rational curve with negative self-intersection on $X$ is a $(-1)$-curve.
This implies that every other reduction to $\P^2$ is the contraction of $r$
disjoint $(-1)$-curves. In particular,
$\G = \{y_1,\dots,y_r\}$ is a subset of $\P^2$ of $r$ distinct points.

Let $B$ be an arbitrary integral curve on $\P^2$.
We need to check that $B$ satisfies the condition given in the definition of very general
position for $\G$. Let $C \subset X$ be the proper transform of $B$ via $g$.

If $C$ is $f$-exceptional, then it is a $(-1)$-curve.
In this case, let $U \subset (\P^2)^r$ be a contractible analytic open
neighborhood of $(y_1,\dots,y_r)$. We can assume that every
$u = (y_{1,u},\dots,y_{r,u}) \in U$ is an $r$-ple of distinct points.
Let $g_u \colon X_u \to \P^2$
be the blow-up at $\{y_{1,u},\dots,y_{r,u}\}$. We obtain a flat
family over $U$ with fibers $X_u$.
By Wi\'sniewski's theorem on the constancy of nef-values \cite{Wis},
every $(-1)$-curve of $X$ deforms in the family,
and in particular, so does $C$. Let $C_u \subset X_u$ denote the deformation of $C$,
and let $B_u := g_u(C_u)$. Note that $p_g(C_u) = 0$ for every $u$.
Moreover, if $F_{i,u} := g_u^{-1}(y_{i,u})$,
then $C_u\.F_{i,u}$ is independent of $u$. This means that
$\mult_{y_{i,u}}(B_u)$ is independent of $u$, and we also conclude that
$p_g(B_u)$ is independent of $u$.
By Remark~\ref{rmk:very-gen-pos:analytic}, this implies
that $B$ satisfies the condition in the definition of very general position.

Suppose now that $C$ is not $f$-exceptional, and let $D := f(C)$.
Since $\S$ is in very general position, $(D,(x_1,\dots,x_r))$
moves in a family $\{(D_t,(x_{1,t},\dots,x_{r,t})) \mid t \in T\}$
satisfying the conditions of Definition~\ref{defi:very-gen-pos}.
Using again Remark~\ref{rmk:very-gen-pos:analytic}, we assume that
$T$ is a contractible space.

We can assume that $\{x_{1,t},\dots,x_{r,t}\}$ is a set of distinct points
for every $t \in T$. Let $f_t \colon X_t \to \P^2$ be the corresponding blow-up.
We obtain an algebraic family $F \colon \X \to \P^2 \times T$ over $T$.
By the aforementioned result of Wi\'sniewski, every $(-1)$-curve of $X$
deforms in the family, and since the intersection product is constant in the
family, if two $(-1)$-curves are disjoint on some $X_t$,
then they remain disjoint throughout the deformation over $T$.

Since every reduction to $\P^2$ corresponds exactly to a selection
of $r$ disjoint $(-1)$-curves, the reduction $g \colon X \to \P^2$
also deforms in the family,
thus defining a morphism $G \colon \X \to \P^2 \times T$ over $T$,
inducing, for every $t \in T$, a reduction $g_t \colon X_t \to \P^2$.
Let $\{y_{1,t},\dots,y_{r,t}\}$ be the center of the blow-up of $g_t$.

Fix a general $t \in T$. An arbitrary small deformation
of the points $(y_{1,t},\dots,y_{r,t})$ induces,
via $f\o g^{-1}$, a small deformation of $(x_{1,t},\dots,x_{r,t})$,
which we can assume to
belong to the family parametrized by $T$. This implies that the map
$T \to (\P^2)^{(r)}$ sending $t$ to $\{y_{1,t},\dots,y_{r,t}\}$
has dense image.

For every $t$, denote by $C_t$ the proper transform of $D_t$ by $f$,
and let $B_t := g_t(C_t)$.
Since the family $\{(D_t,(x_{1,t},\dots,x_{r,t}))\mid t \in T\}$
satisfies the conditions of the definition of very general position,
$p_g(C_t)$ is independent of $t$. Note also that
if $F_{i,t}= g_t^{-1}(y_{i,t})$, then $C_t\.F_{i,t}$ is independent of $t$.
We conclude that $p_g(B_t)$ and $\mult_{y_{i,t}}(B_t)$ are independent of $t$.
\end{proof}

\begin{cor}\label{cor:all-reductions}
Let $X$ be the blow-up of $\P^2$ at a set of $r \ge 3$ points
in very general positions.
Let $E$ be a $(-1)$-curve on $X$, and let $X \to Y$ be the contraction of $E$.
Then $Y$ is the blow-up of $\P^2$ at a set of $r - 1$ points
in very general positions.
\end{cor}

\begin{proof}
By Proposition~\ref{prop:very-gen-position},
we only need to check that there is a reduction from $Y$ to $\P^2$.
we know by Proposition~\ref{prop:dF} that $X$ does not contain any
integral rational curve with self-intersection less than $-1$.
It follows that $Y$ does not contain any such curve either.
Therefore, since $Y$ is rational, its minimal models (in the classical sense)
can only be $\P^2$ and $\P^1 \times \P^1$.
The condition that $\rho(Y) = r \ge 3$ is what we need to ensure
that $\P^2$ is indeed a minimal model of $Y$.
\end{proof}

\section{Geography of extremal rays}\label{sect:main}

Let $X$ be the blow-up of $\P^2$
at a set $\S$ of $r$ points in very general position.
We identify the space of numerical classes of 1-cycles
$\big(Z_1(X)/\equiv\big) \otimes \R$
with the N\'eron--Severi space $\big(\Pic(X)/\equiv\big)\otimes \R$,
and denote it by $N(X)$. Let $\NE(X) \subset N(X)$ be the cone spanned by effective
curves, and let $\CNE(X)$, the Mori cone, be its closure.
Let $L$ be the pull-back of a line from $\P^2$.
For any class $\a \in N(X)$, we denote $R(\a) := \R_{\ge 0}(\a)$,
and for any divisor $D$ on $X$, we denote $R(D) := R([D])$.
If $V \subset N(X)$ is a closed cone, then we denote by $V^\o$
the relative interior of $V$,
and set $\de V := V \smallsetminus V^\o$.
Let
$$
Q = Q(X) = \{ \a \in N(X) \mid \a^2 \ge 0, \a\.h > 0\}.
$$
where $h$ is any ample class on $X$.
The dual of $\CNE(X)$ is the nef cone $\Nef(X)$.
There are inclusions $\Nef(X) \subseteq Q \subseteq \CNE(X)$.

\begin{rmk}\label{rmk:extremal-NE-Nef}
By Hodge Index Theorem, $Q$ is defined in suitable coordinated by
the conditions $u_0^2 \ge u_1^2 + \dots + u_r^2$ and $u_0 \ge 0$.
Moreover, if $R$ is a ray supported on the boundary of $Q$,
then $R^\perp \cap Q = R$. This implies that if $R$ is a ray supported on the
boundary of $Q$, then $R\subset\de\CNE(X)$ if and only $R \subset\Nef(X)$,
and of course in this case $R$ is in the boundary of $\Nef(X)$.
Note that such a ray is always an extremal ray of $\Nef(X)$, but not
necessarily of $\CNE(X)$.
\end{rmk}

Within the extremal rays of $\CNE(X)$, we focus on the isolated ones.
We denote the set of isolated extremal rays of $\CNE(X)$ by $\E$.
We also consider the subsets $\E^-,\E^+\subseteq \E$
or rays that are spanned by an integral curve $C$ with
$K\.C < 0$ or $K\.C > 0$, respectively.
By the Cone Theorem, $\E^-$ is precisely the set of rays
spanned by $(-1)$-curves. We call them {\it $(-1)$-rays}.
In the following we will also denote the set of $(-1)$-rays by $\{R_i\}$,
so that we can use the shorted notation $\sum R_i$ to denote
$\sum_{R \in \E^-}R$.

After fixing an Euclidean metric on $N(X)$, we define the {\it angular distance}
$d(V,V')$ between two cones $V,V'\subseteq N(X)$ as in \cite[Section~1]{dF}.
The next property holds in general, on any smooth projective surface.

\begin{lem}[\protect{\cite[Lemma~1.2]{dF}}]\label{lem:dF}
For every extremal ray $R_i$ of $\CNE(X)$,
$$
d(R,Q) \le \inf\{d(R,R') \mid
\text{$R'$ is an extremal ray of $\CNE(X)$ different from $R$}\}.
$$
In particular, any Cauchy sequence of extremal rays of $\CNE(X)$ converges to a ray on
the boundary of $Q$.
\end{lem}

We obtain the following fact.

\begin{prop}\label{prop:E=E+uE-}
With the above notation, $\E = \E^- \cup \E^+$.
\end{prop}

\begin{proof}
Lemma~\ref{lem:dF} implies that an
isolated extremal ray of $\CNE(X)$ cannot lie on $Q$,
and thus is spanned by a class with negative self-intersection.
As explained in \cite[Remark~1.4.33]{Laz}, any extremal ray
that is spanned by a class with negative self-intersection
is in fact spanned by the class of an integral curve, and is isolated.
Proposition~\ref{prop:dF} implies that there are no $(-2)$-curves on $X$.
This excludes the only possible irreducible and reduced curves with
negative self-intersections
and trivial canonical degree, and thus completes the proof.
\end{proof}

To analyze further the geometry of $\CNE(X)$,
the idea is to look at the positions of the $(-1)$-rays
with respect to the cone $Q + R(K)$.
One can visualize the area that is contained in this cone
as the {\it shade} of $Q$ from $-K$, imagining $-K$ as the {\it source of light}.

If $r < 9$, then $X$ is a Del Pezzo surface,
and there are finitely many $(-1)$-rays (marked as black dots in the figure).
These are the only extremal rays of $\CNE(X)$.
Since $-K \in Q^\o$, we have $Q + R(K) = N(X)$, and hence, trivially,
every $(-1)$-ray is in the interior of this region.

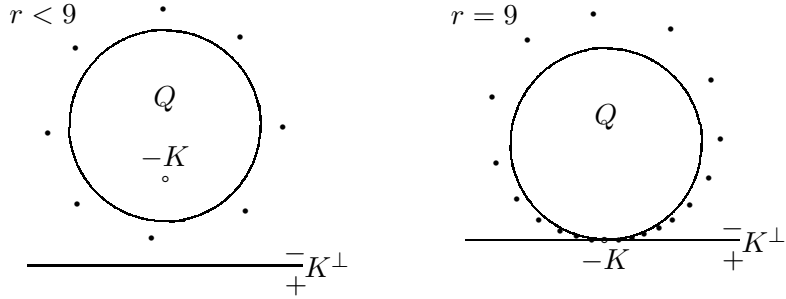
\begin{figure}[h!]
\centering
\input r-up-to-9.pic
\end{figure}

If $r = 9$, then $-K$ lies on the boundary of $Q$, and thus
$Q + R(K) = K^{\le 0}$. All the $(-1)$-rays are contained
in the interior of this half-space.
There are infinitely many $(-1)$-rays,
and they cluster to $R(-K)$ (the latter is always an effective ray,
generated by the proper transform of the unique cubic passing through the nine points).
The cone $\CNE(X)$ is generated
by the $(-1)$-rays and $R(-K)$, and $R(-K)$
is the only non-isolated extremal ray of the cone
(cf. \cite[Proposition~12]{Nag}).

As soon as we blow up ten or more points,
the configuration of the extremal rays begins to show a different behavior.

\begin{prop}\label{thm:(-1)-rays:r>9}
Let $f\colon X \to \P^2$ be the blow-up at a set of
$r \ge 10$ points in very general position, and let $s := \sqrt{r-1}-3$.
Then every $(-1)$-ray of $X$ is contained in the cone $Q + R(K-sH)$, and
if the $(-1)$-Rays Conjecture is true, then for every $r \ge 10$ we have
$$
\CNE(X) |_{(K-sL)^{\ge 0}} = Q |_{(K-sL)^{\ge 0}},
$$
where $L$ is the pull-back of a line in $\P^2$
(under any reduction to $\P^2$).
Moreover:
\begin{enumerate}
\item
If $r = 10$, then every $(-1)$-ray lies on the boundary of $Q + R(K)$ and,
assuming that the $(-1)$-Rays Conjecture is true, we have
$$
\CNE(X) |_{K^{\ge 0}} = Q |_{K^{\ge 0}}.
$$
\item
If $r > 10$, then every $(-1)$-ray lies on the complement of $Q + R(K)$ and
$$
\CNE(X) |_{K^{\ge 0}} \supsetneq Q|_{K^{\ge 0}}.
$$
\end{enumerate}
\end{prop}

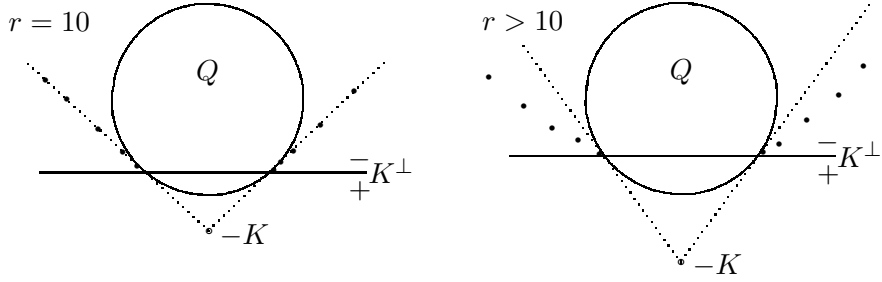
\begin{figure}[h!]
\centering
\input r-at-least-10.pic
\label{fig:2}
\end{figure}

The proof of this proposition is based on elementary computations. 
For the convenience of the reader, we give all the details. 
We start with two lemmas.

\begin{lem}\label{lem:shade:TFAE}
Let $\a, \b \in N(X)$ be classes with $\a^2 < 0$, $\b^2 < 0$ and $\a\.\b < 0$,
and assume that there is a class $\g \in Q^\o$ such that $\a\.\g \le 0$
and $\b \. \g \ge 0$.
Then the following three conditions are equivalent:
\begin{enumerate}
\item[(i)]
$R(\b) \subset (Q + R(\a))$;
\item[(ii)]
$(t\b - \a)^2 = 0$ has solutions;
\item[(iii)]
$\b^\perp \cap Q \subset \a^{\le 0}$.
\end{enumerate}
Similarly, the following three conditions are equivalent:
\begin{enumerate}
\item[(j)]
$R(\b) \subset (Q + R(\a))^\o \cup \{0\}$;
\item[(jj)]
$(t\b - \a)^2 = 0$ has two distinct solutions;
\item[(jjj)]
$\b^\perp \cap Q \subset \a^{< 0} \cup \{0\}$.
\end{enumerate}
\end{lem}

\begin{proof}
We only discuss the first series of equivalences.
First note that for any two nonzero classes $\a,\b\in N(X)$,
we have $R(\b) \subset (Q + R(\a))$ if and only if
$t\b - \a \in Q$ for some $t > 0$. Similarly,
assuming additionally that $\b \not \in Q$,
we have $R(\b) \subset \de(Q + R(\a))$ if and only if
$t\b - \a \in Q$ for a unique $t > 0$.
Therefore one sees immediately that (i) implies (ii), and the reverse implication
follows from the observation that
if $(t\b - \a)^2 = 0$ has solutions, then these are necessarily positive.

Regarding the equivalence of (ii) and (iii), first notice that
$\g \in Q^\o \cap \a^{\le 0} \cap \b^{\ge 0}$, and thus
$Q^\o \cap \a^{<0} \cap \b^{>0} \ne \emptyset$.
Using this, we obtain the following chain of equivalences:
\begin{align*}
\text{$(t\b - \a)^2 \ne 0$ for every $t$}\quad
\Longleftrightarrow\quad
&\text{there is a $\d \in Q^\o$ such that $(t\b - \a) \in \d^\perp$
for every $t$}\\
\Longleftrightarrow\quad
&\text{there is a $\d \in Q^\o$ such that $\d \in \a^\perp \cap \b^\perp$}\\
\Longleftrightarrow\quad
&\text{there is a $\d' \in Q^\o$ such that $\d' \in \a^{<0} \cap \b^{<0}$}\\
\Longleftrightarrow\quad
&\text{$\b^\perp \cap Q \not\subset \a^{\ge 0}$.}
\end{align*}
This finishes the proof of the lemma.
\end{proof}

By duality, there is a one-to-one correspondence to the rays in $\E$ and the extremal
facets of $\Nef(X)$: to any ray $R \in \E$, we associate the facet
$R^\perp \cap \Nef(X)$.

\begin{lem}
Let $R \in \E^-$, and $G = R^\perp \cap \Nef(X)$ be its dual facet.
Assume that $R \subset (Q + R(K))$ (resp., $R \subset (Q + R(K))^\o$).
Then $G \subset K^{\ge 0}$ (resp., $R \subset K^{>0} \cup \{0\}$).
\end{lem}

\begin{proof}
It follows from Lemma~\ref{lem:shade:TFAE} and the fact that
$G \subseteq R^\perp \cap Q$.
\end{proof}

\begin{proof}[Proof of Proposition~\ref{thm:(-1)-rays:r>9}]
We observe that if $C$ is any curve on $X$,
then the discriminant $\D_s$ of the equation $(tC - (K-sL))^2 = 0$ is given by
$$
\D_s/4 = (C\.(K-sL))^2 - C^2(K-sL)^2.
$$
Now, if $s = \sqrt{r-1}-3$, then we have $\D_s \ge 0$ for every
$(-1)$-curve $C$. By Lemma~\ref{lem:shade:TFAE}, this implies that
$R(C)$ is in $Q + R(K-sL)$, which is the first assertion of the proposition.
Similarly, by looking at the discriminant $\D_0$ of the equation
$(tC-K)^2=0$, one sees that if $C$ is a $(-1)$-curve, then
$\D_0 = 0$ if $r=10$ and $\D_0 < 0$ if $r > 10$. This
implies the first assertions in both parts (a) and (b).

Assume then that the $(-1)$-Ray Conjecture is true.
Let $G$ be any facet of $\Nef(X)$. By the $(-1)$-Ray Conjecture,
the dual ray $R$ of $G$ is a $(-1)$-ray, and thus
is contained in $Q + R(K-sL)$ by what just proven.
By Lemma~\ref{lem:shade:TFAE}, $G \subset (K-sL)^{\le 0}$.
Since the boundary of $\Nef(X)$ is supported on the union of the boundary of $Q$ and
at most countably many hyperplanes, it follows that
$$
\Nef(X) |_{(K-sL)^{> 0}}  = Q|_{(K-sL)^{> 0}}.
$$
By continuity, this implies that $\Nef(X) |_{(K-sL)^{\ge 0}} = Q|_{(K-sL)^{\ge 0}}$,
and therefore we have $\CNE(X) |_{(K-sL)^{\ge 0}} = Q|_{(K-sL)^{\ge 0}}$.
This proves the first part of the proposition, and also implies, for $r=10$,
the second assertion in case~(a).

Suppose now that $r > 10$. If $R$ is any $(-1)$-ray of $\CNE(X)$,
as we have seen that $R$ is not contained in $Q + R(K)$,
we have $R^\perp \cap Q \not\subset K^{\le 0}$ by Lemma~\ref{lem:shade:TFAE}.
This implies that
$$
R^{\ge 0} \cap Q^+ |_{K^{\ge 0}}  \subsetneq Q|_{K^{\ge 0}} .
$$
Since $\Nef(X) |_{K^{\ge 0}}  \subseteq R^{\ge 0} \cap Q |_{K^{\ge 0}} $,
it follows that
$\Nef(X) |_{K^{\ge 0}}  \subsetneq Q|_{K^{\ge 0}}$,
and therefore $\CNE(X) |_{K^{\ge 0}} \supsetneq Q|_{K^{\ge 0}}$.
This completes the proof of the proposition.
\end{proof}

\begin{rmk}
To prove parts~(a) and~(b) of the proposition,
one can also give an argument based on the following observation.
Let $E$ be any $(-1)$-curve on $X$,
and let $X \to Y$ be its contraction.
By Corollary~\ref{cor:all-reductions}
$Y$ is the blow-up of $\P^2$ at a set of $r-1$ points in very general position.
The facet $G$ of $\Nef(X)$ dual to $R(E)$ is the pullback of $\Nef(Y)$.
If $r = 10$, then $\Nef(Y) \subset K_Y^{\le 0}$, and thus $G \subset K^{\le 0}$.
However, if $r > 10$, then $\Nef(Y) \subsetneq K_Y^{\le 0}$, and thus
$G \subset K^{\le 0}$.
\end{rmk}

\section{Further remarks on the Mori cone}

As in the previous section, let $X$ be the blow-up of $\P^2$
at a set $\S$ of $r$ points in very general position.
We keep the notation introduced in the previous sections.
In particular, we will denote by $\sum R_i$ the sum of all $(-1)$-rays.

Using a standard degeneration of the points onto an elliptic curve, 
we check the following property. The property (at least, the first inclusion)
is probably well-known to the
specialists; we provide a proof for lack of a reference. 

\begin{lem}\label{thm:weak-inclusion}
For every $r$, we have
$$
\CNE(X) \subseteq R(-K) + Q + \sum R_i = \ov{R(-K) + \sum R_i}.
$$
\end{lem}

\begin{proof}
The statement is clear if $r \le 9$, so we can assume that $r \ge 10$.

To prove the assertions of the lemma, we can pretend that $\S$,
rather than being a set of points in very general position,
is the set of very general points lying on a smooth cubic curve of $\P^2$.
To see this reduction, consider a 1-parameter specialization $\S_\l \to \S_0$
such that $\S_\l$, for very general $\l$,
is a set in very general position, and $\S_0$ is a set of
very general points on a smooth cubic of $\P^2$.
For every $\l$, let $X_\l$ be the blow-up of $\P^2$ at $\S_\l$.
Since $(\P^2)^r$ is rationally connected, we can choose
the specialization to be parametrized by $\l \in \A^1$.
It follows that there is no monodromy, and hence
the spaces $N(X_\l)$ can be naturally identified.
Wi\'sniewski's result on nef values \cite{Wis} implies
that the $(-1)$-rays of $X_\l$ remains constant throughout the deformation.
In other words, under the identification,
$\{R_i\}$ is the set of $(-1)$-rays for every $X_\l$.
Since the class of $-K_{X_\l}$ and the cone $Q(X_\l)$ remain also invariant
under the deformation, if the statement of the theorem holds for $X = X_0$, then
by semicontinuity it holds for $X = X_\l$ for very general $\l \in \A^1$.

So we can assume that the points of $\S$ lie on a smooth cubic.
Let $A \subset X$ be the proper transform of the cubic.
The anticanonical divisor $A$ is the only integral curve with class in $K^{>0}$,
and arguing as in the proof of \cite[Proposition~2.4]{dF}
(we use the fact that the points on the cubic are chosen very generally),
we see that are no $(-2)$-curves on $X$.
We conclude that every integral curve that is neither
$A$ nor a $(-1)$-curve, has class in $Q|_{K^{\le 0}}$.
It follows by the cone theorem that
\begin{equation}\label{eq:2}\tag{$*$}
\CNE(X) \subseteq R(-K_X) + Q|_{K^\perp} + \sum R_i.
\end{equation}
This implies the first inclusion of the theorem.

For every $\ep > 0$, we consider the cone
$$
Q_\ep := \bigcup\{R \subset N(X) \mid \text{$R$ is a ray with $d(R,Q) \le \ep$}\}.
$$
Since the space of rays is (sequentially) compact, Lemma~\ref{lem:dF} implies that
for every $\ep > 0$ there are only finitely many $(-1)$-rays
that are not contained in $Q_\ep$. In particular,
the cone $R(-K) + Q_\ep + \sum R_i$ is closed for every $\ep > 0$.
As
$$
R(-K) + Q + \sum R_i
= \bigcap_{\ep > 0}\big(R(-K_X) + Q_\ep + \sum R_i\big),
$$
this cone is closed, and thus it contains the closure of
$R(-K) + \sum R_i$.

It remains to prove the reverse inclusion, namely,
that $R(-K)+ Q + \sum R_i$ is contained in the closure
of $R(-K) + \sum R_i$.
Since $Q \subseteq \CNE(X)$, we see by \eqref{eq:2}
that it suffices to show that
$$
Q|_{K^\perp} \subseteq \ov{R(-K) + \sum R_i}.
$$
Suppose by way of contradiction that this inclusion does not hold.
Since all $(-1)$-rays are in the half-space $K^{\le 0}$
and $R(-K)$ is in the half-space $K^{\ge 0}$, if the inclusion
does not hold then
$$
Q|_{K^\perp} \not\subseteq \ov{\R[K] + \sum R_i}.
$$
This implies that if $H$ is any fixed ample divisor on $X$, then
$$
Q|_{(K+tH)^\perp} \not\subseteq \ov{\R[K+tH] + \sum R_i} \ \fall 0 < t \ll 1.
$$

Fix any such $t$.
Since $\ov{\R[K+tH] + \sum R_i}$ is a convex set, we can find a class
$\a \in \de Q|_{(K+tH)^\perp}$
that does not belong to it. If we consider the linear projection
$\pi_t \colon N(X) \to (K+tH)^\perp$
with kernel $\ker(\pi_t) = \R[K+tH]$, then this means that
$\a \not\in \ov{\sum \p_t(R_i)}$.
Note also that $\p_t(Q|_{K^\perp}) \subseteq (Q|_{(K+tH)^\perp})^\o$, and so
$$
\a \not\in \p_t(Q|_{K^\perp}) = \p_t(R(-K) + Q|_{K^\perp}).
$$
Since $Q \subseteq \CNE(X)$, we know that the class $\a$ is in the Mori cone.
By \eqref{eq:2}, this implies that
$\a \in R(-K_X) + Q|_{K^\perp} + \sum R_i$. Therefore
$$
\a \in \p_t\big(R(-K) + Q|_{K^\perp} + \sum R_i\big)
= \p_t(Q|_{K^\perp}) + \sum \p_t(R_i).
$$
By convexity, we can separate $\ov{\sum \p_t(R_i)}$ from $\a$ by
a linear function. In other words, we can find a class $\om \in N(X)$ such that
$$
\ov{\sum \p_t(R_i)} \subseteq \om^{\ge 0}
\ \and \
\om\.\a < 0.
$$
We can assume without loss of generality that $H\.\a = - \om\.\a$,
and thus we have
$\a \in (\om + H)^\perp$.
We can write $\a = \b + \g$ with
$0\ne \b \in \p_t(Q|_{K^\perp})$ and $0 \ne \g \in \sum \p_t(R_i)$.
Since $\sum \p_t(R_i) \subseteq (\om + H)^{>0}$, this implies that
$$
\p_t(Q|_{K^\perp}) \not \subseteq (\om+H)^{\ge 0}
$$
Let
$$
m = \min \{a \in \R \mid
\p_t(Q|_{K^\perp}) \subseteq (\om+ a H)^{\ge 0}\},
$$
and fix a nonzero class $\d \in\p_t(Q|_{K^\perp}) \cap (\om+ m H)^\perp$.
Note that $m > 1$.
We consider the other intersection point $\a'$ of the line
$(1-s)\a + s\d$ with $\de Q|_{(K+tH)^\perp}$. Note that
$\a'$ is determined by some $s > 1$, and thus
$(\om + mH)\.\a' < 0$. On the other hand,
both sets $\p_t(Q|_{K^\perp})$ and $\sum \p_t(R_i)$
are contained in $(\om + mH)^{\ge 0}$. It follows that the class $\a'$
cannot be written as a positive linear combination
of elements in these sets. This implies that
$$
\a' \not\in R(-K_X) + Q|_{K^\perp} + \sum R_i.
$$
In view of \eqref{eq:2}, this gives
a contradiction since $\a' \in Q \subseteq \CNE(X)$.
\end{proof}

We consider the linear projection $\pi \colon N(X) \to K^\perp$
with kernel $\ker(\pi) = \R[K]$.
When $r \ge 10$, Lemma~\ref{thm:weak-inclusion} implies that
$$
Q|_{K^\perp} = \p(Q) \subseteq \ov{\sum \p(R_i)}.
$$
This property can equivalently be formulated in terms of extremal facets.

An extremal facet $F$ of $\CNE(X)$ is said to be {\it $K$-negative}
if it is fully contained in the half-space $K^{<0}$.
The Mori contraction of a $K$-negative facet gives either a reduction down to $\P^2$,
or a conic bundle structure over $\P^1$. Every reduction to $\P^2$
corresponds to a simplicial facet with $r$ extremal rays, whereas every
conic bundle structure over $\P^1$
corresponds to a facet spanned by $2(r-1)$ extremal rays,
intersecting the boundary of $Q$ along the ray spanned by the class of
a general fiber.

\begin{prop}\label{cor:K-projection}
Assume that $r \ge 10$. The quadric subcone $Q|_{K^\perp}$
of $K^\perp$ is contained
in the closure of the locus covered by the images, via $\p$,
of all the $K$-negative facets of $\CNE(X)$.
For every $K$-negative facet $F$ of $\CNE(X)$,
every $(r-9)$-dimensional face of $\p(F)$
(which is the image of an $(r-9)$-dimensional face $F'$ of $F$) intersects
$Q|_{K^\perp}$ along a ray $R$. If $X \to Y$ denotes the Mori contraction of $F'$, then
$R$ is the ray spanned by the proper transform $C$
of the unique anticanonical curve of $Y$.
\end{prop}

\begin{proof}
Denoting by $\{F_j\}$ the set of $K$-negative extremal facets,
we have $\sum \p(R_i) = \sum \p(F_j)$, and thus
$Q|_{K^\perp} \subseteq \ov{\sum \p(F_j)}$ by Lemma~\ref{thm:weak-inclusion}.
By Corollary~\ref{cor:all-reductions},
any Mori contraction of an $(r-9)$-dimensional face $F'$ of
a $K$-negative facet $F$ gives a reduction $X \to Y$ where $Y$
is the blow-up of $\P^2$ at $9$ very general points. $Y$ contains a
unique effective anticanonical divisor $A$. The proper transform $C$
of $A$ on $X$ has class in $(R(-K) + F') \cap \de Q$, and conversely,
$(R(-K) + F') \subset C^\perp$. This means that the cone $R(-K) + F'$ intersects
$Q$ precisely along the ray $R(C)$. Projecting to $K^\perp$,
we conclude that $\p(F')$ intersects $\p(Q)$
(which is equal to $Q|_{K^\perp}$) along $R(C)$.
\end{proof}

If $r = 10$, then we obtain the following statement.

\begin{cor}\label{cor:K-projection:r=10}
Assume that $r=10$. Then $\CNE(X) |_{K^\perp} = Q |_{K^\perp}$ and
$$
\p(\CNE(X)) = \p(\Nef(X)) = Q |_{K^\perp}.
$$
Moreover, the $K$-negative part of the boundary $\de \CNE(X)$
is entirely covered by the $K$-negative facets.
These induce, via the projection $\p$,
a decomposition of the quadric cone $Q|_{K^\perp}$
into countably many rational polyhedral cones with extremal rays lying
on the boundary of $Q|_{K^\perp}$, and we have
$$
\p\Big(\big(\de\CNE(X)|_{K^{<0}}\big)\smallsetminus
\big(\bigcup R_i\big)\Big) = Q |_{K^\perp}^\o.
$$
\end{cor}

\begin{proof}
Since by Proposition~\ref{thm:(-1)-rays:r>9} every $(-1)$-ray $R_i$
is in the boundary of $Q + R(K)$, we have $\p(R_i) \in \de(\p(Q))$.
Applying then Lemma~\ref{thm:weak-inclusion}, we obtain
$$
\p(Q) \subseteq \p(\CNE(X)) \subseteq \p(R(-K) + Q + \sum R_i) \subseteq \p(Q),
$$
which implies that this is a chain of equalities.
Note that $\p(Q) = Q|_{K^\perp}$. Since
$$
Q|_{K^\perp} \subseteq \CNE(X)|_{K^\perp} \subseteq \p(\CNE(X)),
$$
we conclude that $\CNE(X) |_{K^\perp} = Q |_{K^\perp}$ and
$\p(\CNE(X)) = Q |_{K^\perp}$.
By Remark~\ref{rmk:extremal-NE-Nef}, we also have
$\Nef(X)|_{K^\perp} = Q|_{K^\perp}$, and thus $\p(\Nef(X)) = Q |_{K^\perp}$.
The last assertion of the corollary follows by Proposition~\ref{cor:K-projection}.
\end{proof}

For $r=10$ the first inclusion
in the statement of Lemma~\ref{thm:weak-inclusion}
amounts to say that any ray in $\E^+$ (if there is any at all)
is contained in the cone $R(-K) + Q$.

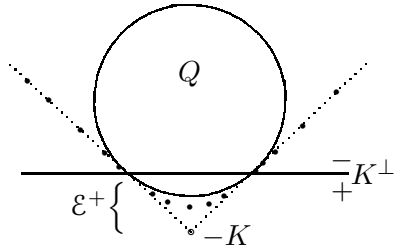
\begin{figure}[h!]
\centering
\input r-10.pic
\end{figure}

This can also be seen from Lemma~\ref{lem:shade:TFAE}.
The argument is analogous to the one in the proof of
Proposition~\ref{thm:(-1)-rays:r>9}, and goes as follows.
Suppose that $C$ is an integral curve
such that $R(C) \in \E^+$, and consider the equation
$(tC+K)^2 = 0$. The discriminant $\D$ is, again, given
$\D/4 = (C\.K)^2 - (C^2)(K^2)$. By the positivity of the arithmetic genus
and the fact that $K^2 = -1$, we see that $\D \ge 0$.

This argument gives something more: the equality $\D = 0$ holds
if and only if $C^2=-1$ and $p_a(C) = 1$.
This means that a ray of $\E^+$ (if any)
lies on the boundary of $Q + R(-K)$ if and only if
it is spanned by an integral curve $C$ with $C^2=-1$ and $p_a(C) = 1$.

\begin{cor}\label{cor:r=10:excluding-C}
If $r = 10$, then $X$ does not contain any integral
curve $C$ with $C^2=-1$ and $p_a(C) = 1$.
\end{cor}

\begin{proof}
Assume that $X$ does contain a curve $C$ with $C^2=-1$ and $p_a(C) = 1$.
Consider the same degeneration as in the beginning of the proof of
Lemma~\ref{thm:weak-inclusion}, and let $C_0$ be the flat limit
(as a divisor) of $C$.
Since there are no integral curves with class in $K_{X_0}^{>0}$
other than the anticanonical divisor $A$,
$C_0$ must contain $A$ as one of its irreducible components.
We have seen that $[C]$ is in the boundary of $Q + R(-K)$.
Similarly, all $(-1)$-rays of $X$ are in the boundary of $Q + R(K)$
(cf. Proposition~\ref{thm:(-1)-rays:r>9}).

We consider two possibilities, according to whether $R(-K)$ and $R(C)$
are aligned with a $(-1)$-ray or not.
If $R(-K)$ and $R(C)$ are not aligned with any $(-1)$-ray, then
we apply Corollary~\ref{cor:K-projection:r=10}, which implies that any
other irreducible component $B_i$ of $C_0$ apart from $A$ has
class $[B_i] \in \de Q|_{K^\perp}$. Any such curve also has $p_a(B_i) = 1$,
and writing $C_0 = mA + \sum n_iB_i$ with $m\ge 1$ and $\sum n_i \ge 1$, we would get
$$
p_a(C) = p_a(C_0) \ge m \, p_a(A) + \sum n_i \, p_a(B_i) \ge 2,
$$
a contradiction.

Therefore $R(-K)$ and $R(C)$ must be aligned with an $(-1)$-ray $R(E)$, and we have
$C \equiv t(E-K) - sK$ for some $t,s \ge 0$.
The condition $C^2 = -1$ implies that $s = 1$, and we deduce that $C\.E = 1$.
Let $g\colon X \to Y$ be the contraction of $E$
($Y$ is the blow-up at $\P^2$ at 9 very general points), and let $D = g(C)$.
We have $D^2 = 0$ and $K_Y\.D = 0$, and thus $D$ is the anticanonical curve of $Y$.
This implies that $C$, being equal to $g^*D - E$, is in $|-K_X|$, which is empty,
again a contradiction.
\end{proof}

We close with a consideration on the cardinality of extremal rays.
We have already mentioned that
the Mori cone of $X$ has finitely many extremal rays when $r \le 8$ and
countably many if $r = 9$. Analogous properties hold for the
nef cone of $X$.

Although certainly expected, the following result seems to be new.
One should contrast this result with the fact that the Weyl group of
$X$ is at most countable. 

\begin{thm}
If $r \ge 10$, then both cones $\CNE(X)$ and $\Nef(X)$
have an uncountable number of extremal rays.
\end{thm}

\begin{proof}
We choose $r-10$ disjoint $(-1)$-curves $E_1,\dots,E_{r-10}$ of $X$,
and consider the associated Mori contraction
$g\colon X \to Y$. By Corollary~\ref{cor:all-reductions},
$Y$ is the blow-up of $\P^2$ at 10 very general points.
We fix any ray
$$
\ov R \subset \de Q(Y)|_{K_Y^\perp}
$$
that is not in the span of $R(-K_Y)$ with any $(-1)$-ray of $Y$.
Let $\ov \a$ be any generator of $\ov R$, and let
$\a := g^*\ov \a \in N(X)$.
We claim that the ray $R := R(\a)$ is an extremal ray of both $\CNE(X)$
and $\Nef(X)$. Since there are uncountably many choices for $\ov R$,
and each gives a different ray $R$, the assertion will follow.

By Corollary~\ref{cor:K-projection:r=10}, $\ov\a$ is a nef class,
and so $\a$ is nef as well. Since $\a^2 = \ov\a^2=0$, and $K\.\a = K_Y\.\ov\a = 0$,
it follows that
$$
R \subset \de Q|_{K^\perp} \cap \Nef(X)
$$
By Remark~\ref{rmk:extremal-NE-Nef}, this implies that $R$
is an extremal ray of $\Nef(X)$, and that it
is contained in the boundary of $\CNE(X)$.
We are left to show that $R$ is also an extremal ray of $\CNE(X)$.
We proceed by induction on $r \ge 10$.

If $r=10$, then $X = Y$ and $R = \ov R$. Suppose we have
$a = \a_1 + \a_2$ for some classes $\a_i\in\CNE(X)$.
In the argument that follows, we make free use of the
conclusions of Corollary~\ref{cor:K-projection:r=10}.
Since
$\a = \p(\a) = \p(\a_1) + \p(\a_2)$,
we must have $\p(\a_1),\p(\a_2) \in R$.
We can assume that $\a_1 \in K^{\le 0}$.
Since $\a$ is in the boundary of $\CNE(X)$, it follows by the convexity
of this cone that both $\a_1$ and $\a_2$ are in the boundary of $\CNE(X)$.
Therefore, if $\a_1$ is not in $R$, then is strictly $K$-negative, and as
it is in the boundary of $\CNE(X)$, it must be contained in
a $K$-negative facet $F$ of $\CNE(X)$.
In particular, we have $\p(\a_1) \in \p(F)$. On the other hand,
by our choice of $\a$, this class is not in the image of any $K$-negative facet
of the Mori cone. Therefore $\a_1$, and hence $\a_2$ too, are in $R$.

Suppose now that $r > 10$, and that the statement is proven for $r-1$.
We consider the contraction $h \colon X \to X'$ of $E_1$,
and let $g' \colon X' \to Y$ be the induced map.
Let $\a' := (g')^*\ov\a$.
By induction we known that $R' := R(\a')$ is an extremal ray of $\CNE(X')$.
Note that $\a = h^*\a'$ and therefore $\a' = h_*\a$.
Suppose by contradiction that $R$ is not an extremal ray of $\CNE(X)$.
Then we can write $\a = \b_1 + \b_2$
where $\b_1,\b_2 \in \CNE(X)$ are two linearly independent classes.
Since $R'$ is extremal in $\CNE(X')$ and $h_*\a \in R'$,
we have $h_*\b_1,h_*\b_2 \in R'$. More precisely, we have
$h_*\b_1 = s\a'$ and $h_*\b_2=(1-s)\a'$ for some $s \in (0,1)$.
As $h_*$ has kernel $\ker(h_*) = \R[E_1]$, this implies that
$$
\b_1 = s\a + t[E]
\ \and \
\b_2 = (1-s)\a - t[E]
$$
for some nonzero $t$, that we can assume to be positive.
As $\b_1 \in \CNE(X)|_{K^{<0}}$, Mori's Cone Theorem implies that
$\b_1 = \g_1 + \g_2$
where $\g_1$ is given by a positive linear combination of classes of $(-1)$-curves, and
$\g_2$ is in $\CNE(X)|_{K^{\ge 0}}$. Using again that $R'$ is an extremal ray,
we see that $h_*\g_1,h_*\g_2  \in R'$. This implies that the classes
$\g_1$ and $\g_2$ belong to the
2-dimensional space generated by $\b_1$ and $\b_2$.
After replacing $\b_1$ by $\g_1$ and changing
$\b_2$ accordingly, we can therefore assume that
when we wrote $\a = \b_1 + \b_2$, we had $\b_1$ given by a positive linear combination
of classes of $(-1)$-curves. Taking push-forward and recalling that
$h_*\b_1 \in R'$, this implies that $\a'$ is a positive linear combination
of classes of $(-1)$-curves of $X'$. This is however impossible, as $K_{X'}\.\a' = 0$.
\end{proof}

\end{document}